\documentclass[11pt,reqno]{amsart}
\usepackage{amssymb,latexsym}
\usepackage{graphicx}
\usepackage{url}	
 \usepackage{setspace}
\calclayout

\makeatletter
\newcommand{\monthyear}[1]{%
  \def\@monthyear{\uppercase{#1}}}
\newcommand{\volnumber}[1]{%
  \def\@volnumber{\uppercase{#1}}}
\AtBeginDocument{%
\def\ps@plain{\ps@empty
  \def\@oddfoot{\@monthyear \hfil \thepage}%
  \def\@evenfoot{\thepage \hfil \@volnumber}}
\def\ps@firstpage{\ps@plain}
\def\ps@headings{\ps@empty
  \def\@evenhead{%
    \setTrue{runhead}%
    \def\thanks{\protect\thanks@warning}%
    \uppercase{The Fibonacci Quarterly}\hfil}%
  \def\@oddhead{%
    \setTrue{runhead}%
    \def\thanks{\protect\thanks@warning}%
    \hfill\uppercase{A System of 4 Recursions}}%
  \let\@mkboth\markboth
  \def\@evenfoot{%
    \thepage \hfil \@volnumber}%
  \def\@oddfoot{%
    \@monthyear \hfil \thepage}%
  }%
\footskip=25pt
\pagestyle{headings}%
}
\makeatother

\newcommand{\s}[2]{#1^{(#2)}}

\theoremstyle{plain}
\numberwithin{equation}{section}
\newtheorem{thm}{Theorem}[section]
\newtheorem{theorem}[thm]{Theorem}
\newtheorem{lemma}[thm]{Lemma}
\newtheorem{example}[thm]{Example}
\newtheorem{definition}[thm]{Definition}
\newtheorem{proposition}[thm]{Proposition}

\begin{document}
\monthyear{Month Year}
\volnumber{Volume, Number}
\setcounter{page}{1}

\title[A System of Four Simultaneous Recursions]{A System of Four simultaneous Recursions \\ Generalization of the Ledin-Shannon-Ollerton Identity }
\author{Russell Jay Hendel}
\address{Towson University}
\email{rhendel@towson.edu}

\begin{abstract}
This paper further generalizes a recent result of Shannon and Ollerton who resurrected an old identity due to Ledin.  
This paper generalizes the Ledin-Shannon-Ollerton result to all the metallic sequences. The results give closed formulas for the sum of products of powers of the first $n$ integers with the first $n$ members of the metallic sequence. 
Three key innovations of this paper are i) reducing the proof of the generalization to the solution of  a system of 4 simultaneous recursions;
ii)  use of the shift operation to prove equality of polynomials; and iii) new OEIS sequences
arising from the coefficients of  the four polynomial
families satisfying the 4 simultaneous recursions.
\end{abstract}

\maketitle

\section{Introduction}
 Shannon and Ollerton \cite{Shannon} in a beautiful paper recently resurrected an old result of Ledin \cite{Ledin}. Here are two identities, one from the Ledin paper and one from this paper.

 \begin{equation}\label{e1}
    \displaystyle \sum_{k=1}^n k^2 F_k = 
     \biggl(n^2-2n+5\biggl) F_n +
    \biggl(n^2-4n+8)\biggr) F_{n+1} -8,
\end{equation}
where $\{F_k\}_{k \ge 0}$ are, as usual the Fibonacci numbers, and $n$ is an arbitrary positive integer.

\begin{equation}\label{e2}
    \displaystyle \sum_{k=1}^n k^3 P_k = 
    \frac{1}{2} \times 
    \biggl(n^3+3n-3\biggl) P_n +\frac{1}{2} \times
    \biggl(n^3-3n^2+6n-7)\biggr) P_{n+1} +\frac{7}{2},
\end{equation}
where $\{P_k\}_{k \ge 0}$ are, as usual the Pell numbers, and $n$ is an arbitrary positive integer.

What intrigued Shannon, Ollerton, and Ledin are the polynomial coefficients in the above equations: 
$n^2-2n+5, n^2-4n+8, \frac{1}{2}  ( n^3+3n-3  ), \text{ and } \frac{1}{2}  ( n^3-3n^2+6n-7 ).$

These examples  naturally motivate generalization. An outline of the rest of this paper is as follows. Section 2 describes the main result of this paper generalizing all previous results. The four, integer, triangular arrays used to produce the polynomial coefficients in \eqref{e1}-\eqref{e2} are presented in Section 3; they naturally correspond to  new OEIS sequences. 
Section 4 discusses the contributions of this paper to previous results.
Sections 5 begins the proof and naturally motivates defining the shift algebraic operator whose properties are explored in Section 6.
Section 7 completes the proof. Section 8 presents fun facts, patterns, and identities on the four new OEIS triangular arrays.

\section{The Main Result}

We first recall the definition of the \emph{metallic} sequences \cite{Spinadel, Wikipedia}.  
\begin{definition} The metallic sequence of order $m \ge 1,$ is defined recursively by
\begin{equation}\label{e3}
    G_0=0, G_1 =1, G_n= m G_{n-1} + G_{n-2}, n \ge 2.
\end{equation}
For $m=1,2$ the sequences have  names, Golden and Silver, corresponding to the Fibonacci and  Pell numbers.  Sequences for $m \ge 3,$ depending on the author, may be called Bronze, or Copper etc. When emphasis is needed, the symbol $G^{(m)}$ will indicate the metallic sequence with parameter $m$ \eqref{e3}. We will use the term 
\emph{metallicity of the sequence} to refer to $m.$
\end{definition}

\textbf{Convention.} Throughout the paper if $r(x)$ is any polynomial of degree $deg(r),$ then the notation for the coefficients of $r(x)$ are given by 
\begin{equation}\label{e4}
r(x) = \sum_{i=0}^{deg(r)} r_i x^i.
\end{equation}

\begin{theorem}[Main Theorem] 
For  $m \ge 1, n \ge 0,p \ge 1$  there exist two degree-p polynomials with rational coefficients, $F(x), T(x)$ such that
\begin{equation}\label{e5}
    S(n,m,p) = \displaystyle \sum_{k=0}^n k^p G^{(m)}_k =
    F(n) G^{(m)}_n + T(n) G^{(m)}_{n+1} - T_0.
\end{equation}
Moreover we can give explicit form to $F(n)$ and $T(n).$ We first define four integer-polynomial families satisfying the following
system of simultaneous recursions.
\begin{multline}\label{e6}
	\s{f}{0}(X) =1; \qquad \s{t}{i}(X) = \sum_{j=0}^{j=i} \binom{i}{j} X^{i-j} \s{f}{j}(X), i \ge 0; 
	\s{d}{i}(X) = \s{t}{i}(X) - X\s{f}{i}(X),  i \ge 0; \\
	 \s{s}{i}(X) = \s{f}{i}(X)+\s{d}{i}(X),  i \ge 0; \qquad \s{f}{i}(X)=\sum_{j=0}^{j=i-1} \binom{i}{j} X^{i-1-j} \s{s}{j}(X), i \ge 1. 
\end{multline}
Then the formulas for $F(n)$ and $T(n)$ are
\begin{equation}\label{e7}
F(n) = \sum_{i=0}^p  \frac{\s{f}{i}(m)}{ m^{i+1}} \times (-1)^i \binom{p}{i} n^{p-i}, \qquad
T(n) = \sum_{i=0}^p \frac{ \s{t}{i}(m)}{m^{i+1}} \times (-1)^i \binom{p}{i} n^{p-i}.
\end{equation}
\end{theorem}

For the proof of the Main Theorem we will also need a polynomial $D(k)$, defined by 
\begin{equation}\label{e8}
D(n) = T(n) - n F(n) =
\sum_{i=0}^p  \frac{\s{d}{i}(m)}{ m^{i+1}} \times (-1)^i \binom{p}{i} n^{p-i},
\end{equation}
where the last equality follows from \eqref{e6}.

\section{Basics about the Main Theorem}

Before proceeding, we present in this section, some basics about the Main Theorem.

\textbf{Comment.} There are many symbols in this theorem and throughout the paper. We have endeavored to make the notation mnemonical to facilitate readability. $F$ and $T$ correspond to the \emph{first} and $\emph{second},$ (that is \emph{the number ``t"wo}) polynomial in \eqref{e5} when read from left to right (or corresponding to increasing indices of $G$). $m$ corresponds to the \emph{metallicity} of the recursion. $p$ corresponds to the \emph{power} to which we raise $k$ in \eqref{e5}.  $d$ corresponds to the \emph{difference} (of $t$ and a multiple of $f$) and
$s$ corresponds to the \emph{sum} (of $f$ and $d$).

The proof of \eqref{e5} will be by induction. We  can immediately prove the base case.

\begin{proposition}[Base Case]\label{basecase} For all $m \ge 1, p \ge 1, S(0,m,p) = 0.$ \end{proposition}
\begin{proof} Clear by the defining recursion, \eqref{e3}, and our conventions about polyomial coefficients \eqref{e4}.
\end{proof}

Each of the four polynomial families, $f,t,d,s,$ naturally gives rise to a triangular integer array 
arising from listing coefficients in ascending  powers of $X.$ These four arrays are presented 
in Tables \ref{tab:f} - \ref{tab:s} which provide conventions on row and column notation for these tables. In these tables, the capital
$T$ mnemonically stands for \emph{triangle}, with the superscript indicating which of the four families of rational functions are being described.

\begin{center}
\begin{table}[ht]
 \begin{small}
\caption
{The table $\s{T}{f}.$ Column and row conventions are illustrated for example by i) $\s{T}{f}_{3,2}=8,$
the coefficient of $X^2$ in $\s{f}{3}(X),$ or by, ii) $\s{f}{1}(X)=2-X.$}
\label{tab:f}
{
\renewcommand{\arraystretch}{1.3}
\begin{center}
\begin{tabular}{||c||c|c|c|c|c|c|c|c||} 
\hline \hline

 \;&$X^0$&$X^1$&$X^2$&$X^3$&$X^4$&$X^5$&$X^6$&$X^7$\\
\hline
$ \s{f}{0}(x)$&$1$&\;&\;&$ $&\;&\;&\;&\;\\
$  \s{f}{1}(X)$&$2$&$-1$&\;&\;&\;&\;&\;&\;\\
$  \s{f}{2}(X)$&$8$&$-4$&$1$&\;&\;&\;&\;&\;\\
$  \s{f}{3}(X)$&$48$&$-24$&$8$&$-1$&\;&\;&\;&\;\\
$  \s{f}{4}(X)$&$384$&$-192$&$80$&$-16$&$1$&\;&\;&\;\\
$  \s{f}{5}(X)$&$3840$&$-1920$&$960$&$-240$&$32$&$-1$&\;&\;\\
$  \s{f}{6}(X)$&$46080$&$-23040$&$13440$&$-3840$&$728$&$-64$&$1$&\;\\
$  \s{f}{7}(X)$&$645120$&$-322560$&$215040$&$-67200$&$16128$&$-2184$&$128$&$-1$\\

\hline \hline
\end{tabular}
\end{center}
}
 \end{small} 
\end{table}
\end{center}

\begin{center}
\begin{table}[ht]
 \begin{small}
\caption
{The table $\s{T}{t}.$
Column and row conventions are illustrated for example by i) $\s{T}{t}_{6,2}=7680,$
the coefficient of $X^2$ in $ \s{t}{6}(X),$ or by, ii) $\s{t}{3}(X)=48-2X^2.$}
\label{tab:t}
{
\renewcommand{\arraystretch}{1.3}
\begin{center}
\begin{tabular}{||c||c|c|c|c|c|c|c||} 
\hline \hline

 \;&$X^0$&$X^1$&$X^2$&$X^3$&$X^4$&$X^5$&$X^6$\\
\hline
$ \s{t}{0}(X)$&$1$&\;&\;&\;&\;&\;&\;\\
$  \s{t}{1}(X)$&$2$&\;&\;&\;&\;&\;&\;\\
$  \s{t}{2}(X)$&$8$&\;&\;&\;&\;&\;&\;\\
$  \s{t}{3}(X)$&$48$&\;&$2$&\;&\;&\;&\;\\
$ \s{t}{4}(X)$&$384$&\;&$32$&\;&\;&\;&\;\\
$  \s{t}{5}(X)$&$3840$&\;&$480$&\;&$2$&$ $&\;\\
$  \s{t}{6}(X)$&$46080$&\;&$7680$&\;&$128$&\;&\;\\
$  \s{t}{7}(X)$&$645120$&\;&$134400$&\;&$4368$&\;&$2$\\

\hline \hline
\end{tabular}
\end{center}
}
 \end{small} 
\end{table}
\end{center}

\begin{center}
\begin{table}[ht]
 \begin{small}
\caption
{The table $\s{T}{d}.$ Column and row conventions are illustrated for example by i) $\s{T}{f}_{5,3}=-960,$
the coefficient of $X^3$ in $\s{d}{5}(X),$ or by, ii) $\s{d}{0}(X)=1-X.$}
\label{tab:d}
{
\renewcommand{\arraystretch}{1.3}
\begin{center}
\begin{tabular}{||c||c|c|c|c|c|c|c|c||} 
\hline \hline
 \;&$X^0$&$X^1$&$X^2$&$X^3$&$X^4$&$X^5$&$X^6$&$X^7$\\
\hline
$ \s{d}{0}(X)$&$1$&$-1$&\;&\;&\;&\;&\;&\;\\
$ \s{d}{1}(X)$&$2$&$-2$&$1$&\;&\;&\;&\;&\;\\
$ \s{d}{2}(X)$&$8$&$-8$&$4$&$-1$&\;&\;&\;&\;\\
$ \s{d}{3}(X)$&$48$&$-48$&$26$&$-8$&$1$&\;&\;&\;\\
$ \s{d}{4}(X)$&$384$&$-384$&$224$&$-80$&$16$&$-1$&\;&\;\\
$ \s{d}{5}(X)$&$3840$&$-3840$&$2400$&$-960$&$242$&$-32$&$1$&\;\\
$\s{d}{6}(X)$&$46080$&$-46080$&$30720$&$-13440$&$3968$&$-728$&$64$&$-1$\\

\hline \hline
\end{tabular}
\end{center}
}
 \end{small} 
\end{table}
\end{center}

\begin{center}
\begin{table}[ht]
 \begin{small}
\caption
{The table $\s{T}{s}.$ Column and row conventions are illustrated for example by i) $\s{T}{s}_{3,0}=96,$
the coefficient of $X^0$ in $\s{t}{3}(X),$ or by,  ii) $\s{s}{0}(X)=2-X.$}
\label{tab:s}
{
\renewcommand{\arraystretch}{1.3}
\begin{center}
\begin{tabular}{||c||c|c|c|c|c|c|c|c||} 
\hline \hline

\;&$X^0$&$X^1$&$X^2$&$X^3$&$X^4$&$X^5$&$X^6$&$X^7$\\
\hline
$\s{s}{0}(X)$&$2$&$-1$&\;&\;&\;&\;&\;&\;\\
$\s{s}{1}(X)$&$4$&$-3$&$1$&\;&\;&\;&\;&\;\\
$\s{s}{2}(X)$&$16$&$-12$&$5$&$-1$&\;&\;&\;&\;\\
$\s{s}{3}(X)$&$96$&$-72$&$34$&$-9$&$1$&\;&\;&\;\\
$\s{s}{4}(X)$&$768$&$-576$&$304$&$-96$&$17$&$-1$&\;&\;\\
$ \s{s}{5}(X)$&$7680$&$-5760$&$3360$&$-1200$&$274$&$-33$&$1$&\;\\
$\s{s}{6}(X)$&$92160$&$-69120$&$44160$&$-17280$&$4696$&$-792$&$65$&$-1$\\

\hline \hline
\end{tabular}
\end{center}
}
 \end{small} 
\end{table}
\end{center}

\begin{example} We illustrate \eqref{e5}-\eqref{e7} and Tables \ref{tab:f}-\ref{tab:s} by deriving the polynomial $T(n)$ in \eqref{e2}.
In \eqref{e2}, $p=3,m=2.$ By Table \ref{tab:t}, we have
$\s{t}{0}(X)=1, \s{t}{1}(X)=2, \s{t}{2}(X)=8, \s{t}{3}(X)=48 + 2 X^2.$  
Hence, by \eqref{e7}
$$
	T(n) = \frac{1}{2^1}\binom{3}{0} n^3 - \frac{2}{2^2}\binom{3}{1} n^2 + \frac{8}{2^3} \binom{3}{2}n - \frac{48+2\cdot 2^2}{2^4} 
	\binom{3}{3} n^0  
 	= \frac{1}{2} n^3 - \frac{3}{2} n^2 + \frac{6}{2}n - \frac{7}{2}.
$$
\end{example}

The four triangular arrays have many obvious patterns; these will be explored in Section 8.

For the proof of the Main Theorem we need the following identity.
\begin{proposition}\label{fequald}For $i \ge 1,$
$$	\s{f}{i}(X) = \sum_{j=0}^i \binom{i}{j} X^{i-j} \s{d}{j}(X).$$ 
\end{proposition}
\begin{proof}
To make the notation clearer, we omit the arguments of polynomials.
By \eqref{e6}, $\s{s}{i} = \s{d}{i} + \s{f}{i}.$ Hence, by \eqref{e6}, for $i \ge 1,$
\begin{equation}\label{temp1}
	\s{f}{i}=\sum_{j=0}^{j=i-1} \binom{i}{j}  X^{i-1-j} \s{s}{j} = \sum_{j=0}^{j=i-1} \binom{i}{j} X^{i-1-j} (\s{f}{j} + \s{d}{j}).
\end{equation}
Again, by \eqref{e6}, for $i \ge 1,$
\begin{equation}\label{temp2}
	\s{t}{i}=\sum_{j=0}^{j=i} \binom{i}{j} X^{i-j} \s{f}{j}  =  \s{f}{i} + \sum_{j=0}^{j=i-1} \binom{i}{j} X^{i-j} \s{f}{j}.
\end{equation}
By \eqref{e6}, $\s{d}{i}=  \s{t}{i} - X\s{f}{i}.$ Substituting \eqref{temp1} and \eqref{temp2} we have
\begin{equation}\label{temp3}
	\s{d}{i}= \s{f}{i} + \sum_{j=0}^{j=i-1} \binom{i}{j} X^{i-j} \s{f}{j} - \sum_{j=0}^{j=i-1} \binom{i}{j} X^{i-j} (\s{f}{j} + \s{d}{j})
\end{equation}
The proposition requires us to prove 
\begin{equation}\label{temp4}
	\s{f}{i}= \sum_{j=0}^i \binom{i}{j} X^{i-j} \s{d}{j} = \s{d}{i} +  \sum_{j=0}^{i-1} \binom{i}{j} X^{i-j} \s{d}{j}
\end{equation}
Substituting 
\eqref{temp3}
into 
\eqref{temp4}
 and cancelling $\s{f}{i}$ from both sides
of the resulting equation, we see we must prove
$$
0=\sum_{j=0}^{j=i-1} \binom{i}{j} X^{i-j} \s{f}{j} - \sum_{j=0}^{j=i-1} \binom{i}{j} X^{i-j} (\s{f}{j} + \s{d}{j}) +  \sum_{j=0}^{i-1} \binom{i}{j} X^{i-j} \s{d}{j}.
$$
This completes the proof. 
\end{proof}

\section{Contributions of this paper}

There are four main contributions of this paper over previous results.

First, the obvious contribution, that we are generalizing the Ledin and Shannon-Ollerton results to all metallic sequences.

Second, this generalization introduces new techniques, the technique of simultaneous systems of recursions. Simultaneous systems
have not been explored that much in the Fibonacci Quarterly and they are a welcome venue for future researchers. Additionally, this
paper uses  the shift operator to prove equality of polynomials.

Third, while the Shannon-Ollerton paper delightfully connects the proofs to known and established OEIS sequences, this paper 
leads to four new integer triangular arrays.

Finally, the Shannon-Ollerton approach while also  using an inductive approach
reduced the proof to certain identities with the Bernoulli numbers. The inductive proof in this paper avoids reduction to the Bernoulli numbers. It was hoped that since the proof in this paper avoids the Bernoulli numbers that the 
various conjectures about Bernoulli numbers made by Shannon-Ollerton could be proven by the results of this paper; but so far a straightforward proof has
eluded me.

\section{The Bar Algebraic-Shift Operator}

Recall, that we will  prove the Main Theorem,  by showing that $F(n), T(n)$ as defined by
\eqref{e7} satisfy \eqref{e5}. To accomplish this proof we will need an algebraic operator that shifts arguments of a target polynomial. Since the need for this algebraic operator arises naturally in the proof, we motivate the need, and then define this operator, by beginning the proof of the Main Theorem in this section.

To begin the proof, we assume $m$ and $p$ fixed, allowing us to smoothen the exposition by
 writing $S(n)$ instead of $S(n,m,p)$ and $G_n$ instead of 
$\s{G}{m}_n.$   

The proof is by induction; the base case, $n=0,$ has already been done in Proposition \ref{basecase}. For an induction step we assume \eqref{e5} true  for the case $n$ and proceed to prove it for the case $n+1.$ That is we must prove

\begin{equation}\label{e10}
		S(n+1) = F(n+1) G_{n+1} + T(n+1) G_{n+2} - T_0. 
\end{equation}

But by the induction assumption, \eqref{e5},
$$
		S(n+1) = S(n) + (n+1)^p G_{n+1} = F(n) G_{n} + T(n) G_{n+1} - T_0 + (n+1)^p G_{n+1}.
$$
By \eqref{e3} we can write $G_n = G_{n+2} - m G_{n+1},$ reducing the last equation to the equivalent, 
\begin{equation}\label{e11} 
					S(n+1) =(T(n) - mF(n) G_{n+1} + F(n) G_{n+2} -T_0 +(n+1)^p G_{n+1}=
					D(n) G_{n+1} + F(n) G_{n+2} - T_0 +(n+1)^p G_{n+1},
\end{equation}
the last  equality arising from \eqref{e8}. 

Equating the right hand sides \eqref{e10} and \eqref{e11}, we see that to accomplish the induction step it suffices to prove
\begin{equation}\label{e13}
 	F(n+1) G_{n+1} + T(n+1) G_{n+2} = D(n) G_{n+1} + F(n) G_{n+2}  +(n+1)^p G_{n+1}.
\end{equation}

To prove \eqref{e13} it suffices to equate coefficients of $G_{n+1}, G_{n+2}.$ That is it suffices to prove
\begin{equation}\label{e14}
	F(n+1) = D(n) + (n+1)^p; \qquad T(n+1) = F(n). 
\end{equation}

At this point we get stuck. The traditional way of proving the equality of polynomials, $F(n+1) = D(n),$ by equating corresponding coefficients in the polynomials, does not work naturally here,
since the values of the arguments on the two sides are different ($n+1$ vs. $n$); equating coefficients is therefore not justifiable.

This motivates the following definition. For a given polynomial $R$ define $\bar{R}$ as the polynomial such that
\begin{equation}\label{e15}
	\bar{R}(x+1) = R(x).
\end{equation}
The bar operator is simply a shift operator.
  
Prior to proving the existence and the form of the bar-shift operator in the next section, we show how its existence simplifies the proof.
To prove \eqref{e14}, it suffices, using \eqref{e15}, to prove
\begin{equation}\label{e16}
	F(n+1) = D(n)+ (n+1)^p = \bar{D}(n+1)+ (n+1)^p, \qquad T(n+1) = F(n)  = \bar{F}(n+1).
\end{equation}
In turn, to prove \eqref{e16}, and hence to complete the proof of the Main Theorem, it suffices to prove the two polynomial equalities of
\begin{equation}\label{e17}
	F(X) = \bar{D}(X) +X^p ; \qquad T(X) = \bar{F}(X),
\end{equation}
by showing that corresponding coefficients are equal.

\section{Properties of the Bar Shift Operator}

To prove \eqref{e17} we need properties of the bar-shift operator.

\begin{lemma} $\bar{R}(X)$ always exists, \end{lemma}
\begin{proof} Clear. Using \eqref{e4}, define $S(Y) = \sum_{i=0}^{degree(R)} R_i (Y-1)^i.$
Then $\bar{R}(X)=S(X+1).$  \end{proof}

\begin{example} Let $R(X)=X^2.$ Then $\bar{R}(X)=X^2-2X+1.$  We may verify that
$\bar{R}(X+1) = (X+1)^2 - 2(X+1)+1 = X^2=R(X)$ as required by \eqref{e15}. \end{example}

As just pointed out, by writing $R(X)=R((X+1)-1)$ we can obtain the coefficients of $\bar{R}(X)$ from those of $R(X).$ For a general polynomial $R$ of degree $q$ we have, using our conventions about polynomial coefficients, \eqref{e3}, that 

$\begin{pmatrix} \bar{R}_{q} \\
                \bar{R}_{q-1} \\
                \bar{R}_{q-2} \\
                \vdots \\
                \bar{R}_{0} \\
\end{pmatrix} 
=
\begin{pmatrix}
1 & 0  & 0 &\dotsc &0 \\
-\binom{q}{1} & \binom{q-1}{0} &0  &\dotsc & 0 \\
\binom{q}{2} & \binom{q-1}{1} & \binom{q-2}{0} &\dotsc & 0 \\
\vdots 	      & \vdots  	      &     \vdots   	      &     \vdots   & \vdots 	      \\
(-1)^q \binom{q}{q} & (-1)^{q-1} \binom{q-1}{q-1} & (-1)^{q-2} \binom{q-2}{q-2} &\dotsc &1 
\end{pmatrix}
\begin{pmatrix} R_{q} \\
                R_{q-1} \\
                R_{q-2}\\
                \vdots \\
                R_{0}
\end{pmatrix} $

By equating polynomial coefficients, this matrix equation 
generates $q$ equations.
\begin{equation}\label{e18}
\bar{R}_{q} = R_q; \qquad 
\bar{R}_{q-1} = -\binom{q}{1} R_q +\binom{q-1}{0} R_{q-1}; 
\bar{R}_{q-2}=\binom{q}{2} R_q -\binom{q-1}{1} R_{q-1}+ \binom{q-2}                                                                                                                                                                                                              {0} R_{q-2}; \dotsc
\end{equation}

We close this section with the following obvious observation.
\begin{lemma}\label{lem:linear}
    The polynomial shift operator is a linear
    operator (on polynomials).
\end{lemma}
\begin{proof} Clear. \end{proof}

\section{Completion of the Proof of the Main Theorem }

We continue the proof of the Main Theorem begun in the   Section 5. Recall we showed
that to prove \eqref{e7} satisfies \eqref{e5} it suffices to prove \eqref{e17}. Notice that the polynomial $X^p$ on the right-hand side of the first equation in \eqref{e17} only contributes to the coefficient of $X^i$ when $i=p.$ Accordingly we must deal with two cases. For each case we must prove coefficient equality of the two asserted identities in \eqref{e17}.

\textbf{Case $X^i, i=p.$} Equation \eqref{e17} is stated using the bar-shift operator. But by \eqref{e18} $\bar{D}_p=D_p$ and $\bar{F}_p=F_p.$ Thus it suffices to prove that 
$$ F_p=D_p+1, T_p=F_p.$$
These equations are immediately proven using \eqref{e6}- \eqref{e8}, which state 
$$ F_p=\frac{1}{m}, 
T_p=\frac{1}{m}, \text{ and }
D_p = \frac{1-m}{m}.$$

\textbf{Case $X^{p-i},$ with $i<p.$}

First we prove the second equation in \eqref{e17},
$$
    T_{p-i}= \bar{F}_{p-i}.
$$
By \eqref{e7}
$$ T_{p-i} =(-1)^i \binom{p}{i} \frac{\s{t}{i}(m)}{m^{i+1}}.$$
By \eqref{e6} we therefore have
\begin{equation}\label{e19}
 T_{p-i} =(-1)^i \binom{p}{i} \frac{1}{m^{i+1}} 
 \sum_{j=0}^i \binom{i}{j} m^{i-j} \s{f}{j}(m).
\end{equation}

By \eqref{e18}
$$    
    \bar{F}_{p-i} =
    \sum_{j=0}^i
    \binom{p-i}{i-j}(-1)^{i+j} F_{p-j}.
$$    
Applying \eqref{e7} to this last equation we obtain 
\begin{equation}\label{e20}
 \bar{F}_{p-i} = \sum_{j=0}^i
    \binom{p-i}{i-j}(-1)^{i+j} \binom{p}{j} \s{f}{j}(m)(-1)^j
    \frac{1}{m^{j+1}}.
\end{equation}

To prove the second equation in \eqref{e17} true, we must show that the left hand sides of \eqref{e19} and \eqref{e20} are equal which follows since the right hand sides of \eqref{e19} and \eqref{e20} are equal, where we have used the binomial coefficient identity, \cite{Wolfram},
\begin{equation}\label{e23}
    \binom{p}{i} 
    \binom{i}{j} =
    \binom{p-j}{i-j}
    \binom{p}{j}.
\end{equation}

Next, we prove the first equation in \eqref{e17},
$$
    F_{p-i}= \bar{D}_{p-i}.
$$
By \eqref{e7}
$$ F_{p-i} =(-1)^i \binom{p}{i} \frac{\s{f}{i}(m)}{m^{i+1}}.$$
By Proposition \ref{fequald}  we therefore have
\begin{equation}\label{e21}
 F_{p-i} =(-1)^i \binom{p}{i} \frac{1}{m^{i+1}} 
 \sum_{j=0}^i \binom{i}{j} m^{i-j} \s{d}{j}(m).
\end{equation}

By \eqref{e18}
$$    
    \bar{D}_{p-i} =
    \sum_{j=0}^i
    \binom{p-i}{i-j}(-1)^{i+j} D_{p-j}.
$$    
By \eqref{e6} and \eqref{e8} we further have
\begin{equation}\label{e22}
  \bar{D}_{p-i} = \sum_{j=0}^i
    \binom{p-i}{i-j}(-1)^{i+j} \s{d}{j}(m)(-1)^j
    \frac{1}{m^{j+1}}
    \binom{p}{j}.
\end{equation}

Comparing the right-hand sides of \eqref{e21}-\eqref{e22} we see, by \eqref{e23}, that they are equal, and hence their left-hand sides are equal, completing the proof of the Main Theorem.

\section{Fun Patterns and Identities}

The four, integer, triangular arrays presented in
Tables \ref{tab:f}-\ref{tab:s} are new, not previously found in OEIS. Consistent, with Fibonacci Quarterly tradition, we list a collection of patterns and identities found in  these triangles.
Patterns exist both in the columns and diagonals. We suffice with proving one of these since the proofs are all similar and follow from the defining equation, \eqref{e6}.  All identities hold for $i,j \ge 0,$ unless otherwise stated.

\begin{itemize}
    \item $sign(\s{T}{f}_{i,j}) =  
    sign(\s{T}{d}_{i,j}) = 
    sign(\s{T}{s}_{i,j})=(-1)^j$ ;\; \; 
    $sign(\s{T}{t}_{i,j}) = 1 $
    \item $\s{T}{f}_{i,i} =(-1)^i$ ;\; \;  
    $\s{T}{d}_{i,i+1}=\s{T}{s}_{i,i+1}=(-1)^{i+1} ;\; \;$
    $\s{T}{t}_{i,2j+1}=0$
    \item  $\s{T}{t}_{2i+1,2i}=2;\; \;
        \s{T}{t}_{2i+2,2i}=2^{2i+3}$
    \item $\s{T}{f}_{i,0} = 
    \s{T}{t}_{i,0}=
    \s{T}{d}_{i,0}= -\s{T}{d}_{i,1} =
    \frac{1}{2} \s{T}{s}_{i,0}=
    2^i i!$
    \item For $i \ge 1,$ 
    $\s{T}{f}_{i,1} =2^{i-1} i! ;\; \;
    \text{ for } i \ge 1, 
    \s{T}{s}_{i,1} = -3i!2^{i-1}$ 
    \item $\s{T}{d}_{i,i}=(-1)^i 2^i ;\; \; 
    \s{T}{f}_{i+1,i}=(-1)^i 2^{i+1} ;\; \;
    \s{T}{s}_{i,i}=(-1)^i (2^i+1)$
\end{itemize}

\begin{proof} We prove the fourth bulleted item. First note that by the construction of Tables \ref{tab:f}-\ref{tab:s}, for any symbol $h \in \{f, t, d, s \},$ by \eqref{e4},
$$
        \s{T}{h}_{i,j} = \s{h}{i}_j.
$$
Hence to prove  $\s{T}{f}_{i,0} = 2^i i!$ it suffices to prove
$\s{f}{i}_0 = 2i \s{f}{i-1}_0, i \ge 1.$ 

However, by \eqref{e6}, for $i \ge 1,$ we have $\s{f}{i-1}_0=\s{t}{i-1}_0=\s{d}{i-1}_0.$ Additionally, by \eqref{e6}, since 
$\s{s}{i-1} = \s{f}{i-1} +\s{d}{i-1},$ we have
$\s{s}{i-1}_0=2\s{f}{i-1}_0.$ 

Finally, by \eqref{e6}, we have
$$\s{f}{i}_0 = \binom{i}{i-1} \s{s}{i-1}_0 = 2i \s{f}{i-1}_0$$
as was required to prove.
\end{proof}

\section{Conclusion}

This paper further generalizes the result of Shannon and Ollerton who in turn resurrected an oldie of Ledin. Besides the four new triangular arrays the proof was greatly facilitated by creating a system of four simultaneous recursions two of which never enter the statement of the Main Theorem. We believe this approach of simultaneous systems fruitful and applicable to other areas.

\medskip

\noindent MSC2020: 11B39

\end{document}